\documentclass[12pt]{article}
\usepackage{e-jc}
\usepackage{amsfonts,amssymb,amsmath,amsthm}
\usepackage{pstricks,pst-node,pst-text,pst-3d,pst-plot,pst-tree}
\usepackage{bm}          %% For bold math symbols
\usepackage{graphicx}
\usepackage{url}
\usepackage{pifont}

\usepackage[colorlinks=true,citecolor=black,linkcolor=black,urlcolor=blue]%
    {hyperref}

\newcommand{\arxiv}[1]{\href{http://arxiv.org/abs/#1}{\texttt{arXiv:#1}}}
\newcommand{\seqnum}[1]{\href{http://oeis.org/#1}{#1}}

\newcommand{\N}{{\mathbb N}}
\newcommand{\bone}{ {\mathbf 1} }

\theoremstyle{plain}
\newtheorem{theorem}{Theorem}

\newtheorem{corollary}[theorem]{Corollary}
\newtheorem{proposition}[theorem]{Proposition}

\theoremstyle{definition}
\newtheorem{definition}[theorem]{Definition}

\theoremstyle{remark}

\newcommand{\stirlingsubset}[2]{\genfrac{\{}{\}}{0pt}{}{#1}{#2}}

\newcommand{\associatedstirlingsubset}[2]%
      {\left\{\!\! \stirlingsubset{#1}{#2} \!\! \right\}}
\newcommand{\assocstirlingsubset}[3]%
      {{\genfrac{\{}{\}}{0pt}{}{#1}{#2}}_{\! \ge #3}}

\newcommand{\euler}[2]{\genfrac{\langle}{\rangle}{0pt}{}{#1}{#2}}
\newcommand{\eulergen}[3]{{\genfrac{\langle}{\rangle}{0pt}{}{#1}{#2}}_{\! #3}}
\newcommand{\eulersecond}[2]{\left\langle\!\! \euler{#1}{#2} \!\!\right\rangle}

\newcommand{\eulersecondgen}[3]%
     {{\left\langle\!\! \euler{#1}{#2} \!\!\right\rangle}_{\! #3}}

\newcommand{\binomvert}[2]{\genfrac{\vert}{\vert}{0pt}{}{#1}{#2}}

\newcommand{\nueulergen}[4]%
{{\genfrac{\langle}{\rangle}{0pt}{}{#1}{#2}}^{\! #3}_{\! #4}}

\def\volno{XX}
\def\volyear{20XX}

\title{\bf Generalized Stirling permutations and forests: Higher-order
       Eulerian and Ward numbers}

\author{
J. Fernando Barbero G.\thanks{Partially supported by 
Spanish MINECO grants FIS2012-34379 and FIS2014-57387-C3-3-P.}\\
\small Instituto de Estructura de la Materia, CSIC\\[-0.8ex]
\small Serrano 123, 28006 Madrid, SPAIN\\[2mm]
\small Grupo de Teor\'{\i}as de Campos y F\'{\i}sica Estad\'{\i}stica\\[-0.8ex]
\small Instituto Gregorio Mill\'an, Universidad Carlos III de Madrid \\[-0.8ex]
\small Unidad Asociada al Instituto de Estructura de la Materia, CSIC,
       Madrid, SPAIN \\
\small\tt fbarbero@iem.cfmac.csic.es\\
\and
Jes\'us Salas$\mbox{}^{*,}$\thanks{Partially supported by Spanish MINECO grant 
MTM2011-24097, and by U.S.\ National Science Foundation grant PHY--0424082.}
\qquad  Eduardo J.S. Villase\~nor$\mbox{}^*$\\  
\small Grupo de Modelizaci\'on, Simulaci\'on Num\'erica 
       y Matem\'atica Industrial \\[-0.8ex]
\small Universidad Carlos III de Madrid \\[-0.8ex]
\small Avda.\  de la Universidad, 30\\[-0.8ex]
\small 28911 Legan\'es, SPAIN\\[2mm]
\small Grupo de Teor\'{\i}as de Campos y F\'{\i}sica Estad\'{\i}stica\\[-0.8ex]
\small Instituto Gregorio Mill\'an, Universidad Carlos III de Madrid\\[-0.8ex]
\small Unidad Asociada al Instituto de Estructura de la Materia, CSIC,
       Madrid, SPAIN \\ 
\small\tt \{jsalas,ejsanche\}@math.uc3m.es
}

\begin{document}

% \date{\dateline{submission date}{acceptance date}\\
% \small Mathematics Subject Classifications: comma separated list of
% MSC codes available from http://www.ams.org/mathscinet/freeTools.html}

\date{\dateline{Nov 3, 2014}{XX}\\
\small Mathematics Subject Classifications: 05A05, 05A15, 05C30}

\maketitle
\thispagestyle{empty}   % Suppress page number on front page.

\begin{abstract}
We define a new family of generalized Stirling permutations that can be
interpreted in terms of ordered trees and forests. We prove that the number of
generalized Stirling permutations with a fixed number of ascents is given by 
a natural three-parameter generalization of the well-known Eulerian numbers. 
We give the generating function for this new class of numbers and, in the 
simplest cases, we find closed formulas for them and the corresponding row 
polynomials. By using a non-trivial involution our generalized Eulerian 
numbers can be mapped onto a family of generalized Ward numbers, forming a 
Riordan inverse pair, for which we also provide a combinatorial interpretation. 

% keywords are optional
\bigskip\noindent \textbf{Keywords:} generalized Stirling permutations;
increasing trees and forests; generalized Eulerian numbers; 
generalized Ward numbers
\end{abstract}

%%%%%%%%%%%%%%%%%%%%%%%%%%%%%%%%%%%%%%%%%%%%%%%%%%%%%%%%%%%%%%%%%%
%
% INTRODUCTION
%
\section{Introduction} \label{sec.introduction}

Stirling permutations of order $n$
are permutations of the multiset $\{1^2,2^2,\ldots,n^2\}$
such that, for each $1\le r\le n$, the elements appearing between two
occurrences of $r$ are at least $r$ \cite{Gessel_78}. Given a Stirling
permutation $\rho=r_1r_2\ldots r_{2n}$\,, the index $i$ will be called
an ascent of $\rho$ if $r_i<r_{i+1}$. The number of Stirling permutations
of order $n$ with exactly $k$ ascents is given by second-order Eulerian
numbers $B_{n,k}$ \cite{Gessel_78}. Second order Eulerian numbers are
closely related  to Ward numbers $W_{n,k}$
\cite{Ward_34}, \cite[entry \seqnum{A134991}]{Sloane}. They form an
\textit{inverse pair} in the sense of Riordan \cite{Riordan_68}
(see\cite{Smiley_00}, \cite[entry \seqnum{A008517}]{Sloane}):
\begin{subequations}
\label{Ward_SecondOrderEuler}
\begin{align}
\label{Ward_vs_SecondOrderEuler}
W_{n,k} &\;=\;\sum_{j=0}^k\binom{n-j}{n-k} B_{n,j}\,, \\
B_{n,k} &\;=\;\sum_{j=0}^k(-1)^{k-j}\binom{n-j}{k-j} W_{n,j}\,.
\label{SecondOrderEuler_vs_Ward}
\end{align}
\end{subequations}
We will use Eq.~\eqref{Ward_vs_SecondOrderEuler} to provide a new
combinatorial interpretation of Ward numbers in terms of Stirling permutations.

Stirling permutations and Eulerian numbers have been generalized to
multisets of the form $\{1^\nu,2^\nu,\ldots,n^\nu\}$ with $\nu\in\N$ by Gessel
\cite{Gessel_78a} (as cited by Park \cite{Park_94a}) 
and Park \cite{Park_94a,Park_94b}. Park \cite{Park_94c} related the Stirling 
permutations of these multisets to some generalized Stirling numbers.
(See also \cite[Theorem~1]{Janson_11} and \cite[Theorem~2.1]{Kuba_11}.)
Brenti considered Stirling permutations of the more general multiset
$\{1^{\nu_1},2^{\nu_2},\ldots,n^{\nu_n}\}$ with $\nu_i\in\N$ ($1\le i\le n$) 
in the context of Hilbert polynomials \cite{Brenti_98}; and in relation to 
increasing trees by Kuba and Panholzer \cite{Kuba_11}. The particular case
$\{1^\nu,2^{\nu+2},\ldots,n^{\nu+2}\}$ was also studied by
Janson {\em et al.}\/ \cite{Janson_11}. Stirling permutations 
also appear in the framework of context-free grammars \cite{Chen_93} (see 
Ref.~\cite{Chen_12} for more recent literature). 

The purpose of the paper is to introduce and study natural generalizations of 
the Stirling permutations considered by Gessel and Stanley, Park and other 
authors \cite{Gessel_78,Park_94a}.
We will find bijections between these generalized Stirling permutations
and ordered trees \cite{Park_94a,Janson_11} and forests, providing 
a graph-theoretic interpretation of these objects.  
The combinatorial numbers that count
such Stirling permutations with a fixed number of ascents
are natural generalizations of the Eulerian numbers.
Some particular numbers of this class have been considered
in other contexts \cite{Carlitz_60, Dillon_68, Carlitz_73, Kuba_11} but,
to our knowledge, most of them have not appeared before in the literature.
There are indeed other generalizations of the Eulerian numbers that do not
fall in the above class: e.g., the $r$--Eulerian numbers
\cite{Riordan_58,Foata_70,Magagnosc_80,Bona_12}, or the numbers
$A(r,s\mid \alpha,\beta)$ due to Carlitz and Scoville \cite{Carlitz_74}.

In all our cases, these Eulerian numbers satisfy two-parameter linear
recurrence relations that can be studied in an efficient way by using
generating function techniques \cite{BSV}. With the help of these methods, we
define a family of generalized Ward numbers, and get closed expressions for
them in terms of generalized Eulerian numbers, in the form of inverse pairs
similar to Eqs.~\eqref{Ward_SecondOrderEuler}. These relations provide a
simple combinatorial interpretation for the generalized Ward numbers
in the present context.

%%%%%%%%%%%%%%%%%%%%%%%%%%%%%%%%%%%%%%%%%%%%%%%%%%%%%%%%%%%%%%%%%%
%
% PRELIMINARIES
%
\section{Generalized Stirling permutations}\label{sec.stirling.permutations}

It is useful for our purposes to introduce several definitions based on
the $\nu$--Stirling permutations of order $n$ discussed by Park
\cite{Park_94a}:

\begin{definition} \label{def_new_stirling_permutations}
Let $\nu$ be a positive integer and $X=\{x_1<x_2<\cdots<x_n\}$ be a
totally ordered set of cardinality $n$.
A $(\nu,X)$--Stirling permutation is a permutation of the multiset
$\{x_1^\nu,x_2^\nu,\ldots,x_n^\nu\}$ such that, for each $1\le j\le n$,
the elements occurring between two occurrences of $x_j$ are, at least, $x_j$.
\end{definition}

\medskip

\noindent
{\bf Remarks.} 1. This definition implies that the elements occurring between
two consecutive occurrences of $x_j$ are greater than  $x_j$.
As a consequence of this, the $\nu$ occurrences of $x_n$ have to go together.

2. If $X=[n]$, then the $(\nu,[n])$--Stirling permutations are equivalent
to the canonical $\nu$--Stirling permutations of order $n$. In
Definition~\ref{def_new_stirling_permutations_st}, the $X$ will
correspond to different subsets of $[n]$.

3. If $X=\emptyset$ the unique
$(\nu,\emptyset)$--Stirling permutation is the empty permutation.

\medskip

\begin{definition} \label{def_new_stirling_permutations_t}
Let $\nu,t$ be positive integers, $X=\{x_1<x_2<\cdots<x_n\}$ be a
totally ordered set of cardinality $n$, and consider
$x_0=0 < x_1 < x_2 < \cdots < x_n$.
A $(\nu,t,X)$--Stirling permutation is a permutation of
the multiset  $\{0^t, x_1^\nu,x_2^\nu,\ldots,x_n^\nu\}$ such that for each
$0\le j\le n$ the elements occurring between two occurrences of
$x_j$ are at least $x_j$.
\end{definition}

\medskip

\noindent
{\bf Remarks.} 1. If $t=0$, a $(\nu,0,X)$--Stirling permutation is just a
$(\nu,X)$--Stirling permutation.

2. If $X=\emptyset$ the unique
$(\nu,t,\emptyset)$--Stirling permutation is the permutation $0^t$.

3. The number of $(\nu,t,X)$--Stirling permutations is
$\prod\limits_{k=0}^{|X|-1} (k\nu + t +1)$.

\medskip

In order to count generalized Stirling permutations with a fixed number of
ascents, we introduce a
three-parameter generalization of the standard Eulerian numbers,
that we will refer to as the
\emph{$\nu$-order $(s,t)$-Eulerian numbers}:

\begin{definition} \label{def_gen_Eulerian}
Let $\nu,s\geq 1$ and $t\geq 0$ be integers.
The $\nu$-order $(s,t)$-Eulerian numbers $\nueulergen{n}{k}{(\nu)}{(s,t)}$
are defined as those satisfying the recurrence
\begin{equation}
\nueulergen{n}{k}{(\nu)}{(s,t)} \;=\;
                     (k+s)\, \nueulergen{n-1}{k}{(\nu)}{(s,t)} +
(\nu n - k  + t + 1 -\nu) \, \nueulergen{n-1}{k-1}{(\nu)}{(s,t)}
                           + \delta_{k0}\delta_{n0} \,,
\label{def_recurrence_stnuEulerianOK}
\end{equation} 
with the additional conditions $\nueulergen{n}{k}{(\nu)}{(s,t)}=0$ if $n<0$ or
$k<0$.
\end{definition}

\medskip

\noindent
{\bf Remark.} The values of $\nu,s,t$ do not
have to be integers as $\nueulergen{n}{k}{(\nu)}{(s,t)}$ is obviously a polynomial in these three
parameters. However, we have restricted their ranges to make contact with their
combinatorial interpretation.

\medskip

\begin{proposition} \label{prop.Eulerian_nu_t}
The number of $(\nu,t,[n])$--Stirling permutations with $k$ ascents is equal to
$\nueulergen{n}{k}{(\nu)}{(1,t)}$.
\end{proposition}

\begin{proof}
This is just a generalization of the proof of Eq.~(6.1) in \cite{Dillon_68}.
Let $J_{\nu,t}(n,k)$ be the number of $(\nu,t,[n])$--Stirling permutations
with $k$ ascents. We want to show that
$J_{\nu,t}(n,k)=\nueulergen{n}{k}{(\nu)}{(1,t)}$ by induction on $n$.

The case $n=0$ is trivial:
$J_{\nu,t}(0,k)=\delta_{0,k}=\nueulergen{0}{k}{(\nu)}{(1,t)}$, as there is
a unique permutation of this type (the empty permutation).

Let us assume that $J_{\nu,t}(n-1,k)=\nueulergen{n-1}{k}{(\nu)}{(1,t)}$ for
all $0\le k\le n-1$. We have to insert now the block $n^\nu$. This will
leave the number of ascents unchanged, or increase it
by one unit. We have then only two choices: (1) start from a
$(\nu,t,[n-1])$--Stirling permutation with $k$ ascents, or (2) start from
a $(\nu,t,[n-1])$--Stirling permutation with $k-1$ ascents.
In the first case, we can place the block $n^\nu$ at the beginning of the
permutation or insert it at any of the $k$ ascents.
In the second case, we can insert the block $n^\nu$ at any of the
$\nu(n-1)+t-(k-1)$ non-ascent places.
Hence
$$
J_{\nu,t}(n,k) \;=\; (k+1) J_{\nu,t}(n-1,k) +
                     (\nu n -k +t+1-\nu) J_{\nu,t}(n-1,k-1) \,.
$$
This equation completes the proof.
\end{proof}

\medskip

\noindent
{\bf Remarks.} 1. If $(s,t)=(1,0)$, these numbers reduce to the ordinary
Eulerian numbers for $\nu=1$, to the second-order Eulerian numbers for
$\nu=2$ \cite{Gessel_78}, and to the third-order Eulerian numbers for
$\nu=3$ \cite[entry \seqnum{A219512}]{Sloane}.

2. If $\nu=2$ and $(s,t)=(1,t)$, these numbers correspond to the
generalization by Carlitz \cite{Carlitz_60,Carlitz_73} and
Dillon and Roselle \cite{Dillon_68}.

\medskip

\begin{definition} \label{def_new_stirling_permutations_st}
Let us fix integers $\nu\geq 1$ and  $t\geq 0$, and a generalized ordered
partition $\bm{t}=(t_1,\ldots, t_s)$ of $t$ with $s\geq 1$ parts
(and $t_i\geq0$).
A $(\nu,\bm{t},n)$--Stirling permutation is a sequence
$\bm{\rho}=(\rho_1,\rho_2,\ldots,\rho_s)$, of length $s$, such that each entry
$\rho_i$ is a $(\nu,t_i,X_i)$--Stirling permutation for some generalized
ordered partition  $(X_1,X_2,\ldots,X_s)$  of $[n]$ (where we allow that some
of the $X_i$ are the empty set).
\end{definition}

\medskip

\noindent
{\bf Remarks.} 1. If $\bm{t}=(t)$ (i.e., $s=1$),
the $(\nu,\bm{t},n)$--Stirling permutations
reduce to the $(\nu,t,n)$--Stirling permutations.

2. If $n=0$, there is a single $(\nu,\bm{t},0)$--Stirling permutation:
$(0^{t_1},0^{t_2},\ldots,0^{t_s})$, where in the cases with $t_i=0$ we
have an empty entry.

\medskip

\begin{theorem} \label{theo.Eulerian_nu_st}
The number of $(\nu,\bm{t},n)$--Stirling permutations with $k$
ascents is equal to $\nueulergen{n}{k}{(\nu)}{(s,t)}$.
\end{theorem}

\begin{proof}
Let $J_{\nu,\bm{t}}(n,k)$ be the number of
$(\nu,\bm{t},n)$--Stirling permutations
with $k$ ascents. We want to show that
$J_{\nu,\bm{t}}(n,k)=\nueulergen{n}{k}{(\nu)}{(s,t)}$ by induction
on $n$.

The case $n=0$ is trivial:
$J_{\nu,\bm{t}}(0,k)=\delta_{0,k}=\nueulergen{0}{k}{(\nu)}{(s,t)}$,
as there is a unique permutation of this type:
$(0^{t_1},0^{t_2},\ldots,0^{t_s})$.

Let us assume that
$J_{\nu,\bm{t}}(n-1,k)=\nueulergen{n-1}{k}{(\nu)}{(s,t)}$ for
all $0\le k\le n-1$. Then, as explained in the proof of
Proposition~\ref{prop.Eulerian_nu_t}, we have two choices to insert the
block $n^\nu$ in a  $(\nu,\bm{t},n-1)$--Stirling permutation with $k$
ascents:
(1) start from a  $(\nu,\bm{t},n-1)$--Stirling permutation
with $k$ ascents and
insert the block at the beginning of the $s$ entries or at any of the $k$
ascents; or (2) start from a $(\nu,\bm{t},n-1)$--Stirling permutation
with $k-1$ ascents, and insert the block at any of the
$\nu(n-1)+t-(k-1)$ non-ascent places. Then,
$$
J_{\nu,\bm{t}}(n,k) \;=\; (k+s) J_{\nu,\bm{t}}(n-1,k) +
                     (\nu n -k+ t+1-\nu) J_{\nu,\bm{t}}(n-1,k-1) \,.
$$
This completes the proof.
\end{proof}

\medskip

\noindent
{\bf Remark.} The number of $(\nu,\bm{t},n)$--Stirling permutations with
$k$ ascents  does depend on $\bm{t}$ but only through $t$ and $s$. This is also
true for the number of $(\nu,\bm{t},n)$--Stirling permutations that
is  given by
\begin{equation}
\prod\limits_{k=0}^{n-1} (k\nu + t+s)\,.
\label{eq.number_nu_s_t_n}
\end{equation}

%%%%%%%%%%%%%%%%%%%%%%%%%%%%%%%%%%%%%%%%%%%%%%%%%%%%%%%%%%%%%%%%%%
%
% TREES AND FORESTS
%
\section{Increasing trees and forests} \label{sec.increasing.trees}

Gessel \cite{Gessel_78a}, Park \cite{Park_94a}, and Janson, Kuba and 
Panholzer \cite{Janson_11} discussed the bijection
between $\nu$--Stirling permutations and the class of increasing trees.
In this section we discuss generalizations of these results to the class of
$(\nu,t,[n])$--Stirling permutations introduced above. These latter ones are 
a particular case of the Stirling permutations of the multiset 
$\{1^{\nu_1},2^{\nu_2},\ldots,n^{\nu_n}\}$ discussed by Kuba and Panholzer 
\cite{Kuba_11}. 
For the $(\nu,\bm{t},n)$--Stirling permutations, we introduce a similar
construction in terms of forests.

\begin{definition} \label{def increasing X_tree}
Let $X=\{x_1<\cdots<x_n\}$ be a totally ordered set.
An increasing $X$--tree is a
rooted tree with the internal vertices labelled by the elements of $X$
in such a way that the node labelled $x_1$ is distinguished as the root
and such that, for each $2 \leq i \leq  n$, the labels of the nodes in
the unique path from the root
to the node labelled $x_i$ form an increasing sequence. A generalized increasing
$X$--tree  is an increasing $X_0$--tree with $|X|+1$ internal vertices labelled by the elements of the set
$X_0=\{x_0=0<x_1< x_2< \cdots <x_n\}$.
\end{definition}

\medskip

\noindent{\bf Remark.} The family of generalized increasing $X$--trees is
bijective with the family of increasing $[|X|+1]$--trees.

\medskip

\begin{definition} \label{def d_ary increasing X_trees}
For an integer $d \geq 2$, $d$-ary increasing $X$--trees are increasing
$X$--trees where each internal node has $d$ labelled positions for children.
Equivalently,
for integers $d\geq 2$, $d_0\geq 1$, $(d,d_0)$-ary increasing $X$--trees are
generalized increasing $X$--trees where the root $x_0=0$ has $d_0$ labelled
positions for children, and any non-root internal node $x_i$ ($1\leq i\leq n$)
has $d$ labelled positions for children.
\end{definition}

\medskip

\noindent
{\bf Remarks.} 1. A $d$-ary increasing $X$--tree has $d|X|$ edges,
$|X|$ internal nodes with  outdegree equal to $d$, and $(d-1)|X|+1$ external
nodes.

2. A $(d,d_0)$-ary increasing $X$--tree has $d|X|+d_0$ edges, a root with
outdegree equal to $d_0$, $|X|$ internal nodes with outdegree equal to $d$,
and $(d-1)|X|+d_0$ external nodes.

3. The family of $(d,1)$-ary increasing $X$--trees is bijective with the $d$-ary
increasing $[|X|]$--trees. The family of $(d,d)$-ary increasing $X$--trees is
bijective with the $d$-ary increasing $[|X|+1]$--trees.

\medskip

The following theorem relates the family of $(\nu+1,t+1)$--ary 
increasing $[n]$--trees and the $(\nu,t,[n])$--Stirling permutations. 
The authors independently derived this result, and later discovered 
that this result was already proved in Ref.~\cite[Theorem~2.1]{Kuba_11}. 
See also \cite[Theorem~1]{Janson_11} for a detailed proof of a 
related statement.

\begin{theorem} \label{theo.nu+1.t+1_ary increasing}
Let $\nu\geq 1$, $t\geq 0$ be integers. The family of $(\nu+1,t+1)$--ary
increasing $[n]$--trees is in natural bijection with
$(\nu,t,[n])$--Stirling permutations.
\end{theorem}

\begin{proof}
Our proof is a generalization of Gessel's theorem
(see \cite{Park_94a}) that relies on the argument presented in
\cite{Stanley_86} for ordinary permutations.
Let $\rho$ be any word on the alphabet $\{x_0<x_1<\cdots<x_n\}$ with possible
repeated letters. Let us define a planar tree $T(\rho)$ as follows:
If $\rho=\emptyset$, then $T(\rho)=\emptyset$; if
$\rho\neq\emptyset$, then $\rho$ can be factorized uniquely in the form
$\rho=\rho_1 i\rho_2 i\cdots i \rho_{\nu_i+1}$ where $i$ is the least element
(letter) of $\rho$ and $\nu_i$ its multiplicity. Let $i$ be the root of
$T(\rho)$ and $T(\rho_1)$, $T(\rho_2)$,\dots,
$T(\rho_{\nu_i+1})$ the subtrees (from left to right) obtained by
removing $i$. This yields an inductive definition of $T(\rho)$.
Notice that the outdegree of an internal vertex $i$ is $\nu_i+1$. Notice also
that when $\rho$ corresponds to a generalized Stirling permutation, if $j$
is a letter of $\rho_k$, then $j$ does not belong to any $\rho_l$ for $l\neq k$.
\end{proof}

\medskip

\noindent
{\bf Remark.} See Figure \ref{Fig1} for some simple examples of the Stirling
permutations and their associated trees. 

%%%%%%%%%%%%%%%%%%%%%%%%%%%%%%%%%%%%%%%%%%%%%%%%%%%%%%%%%%%%%%%%%%%%%%%%%%%
%
% FIGURE 1:
%
\begin{figure}[htb]
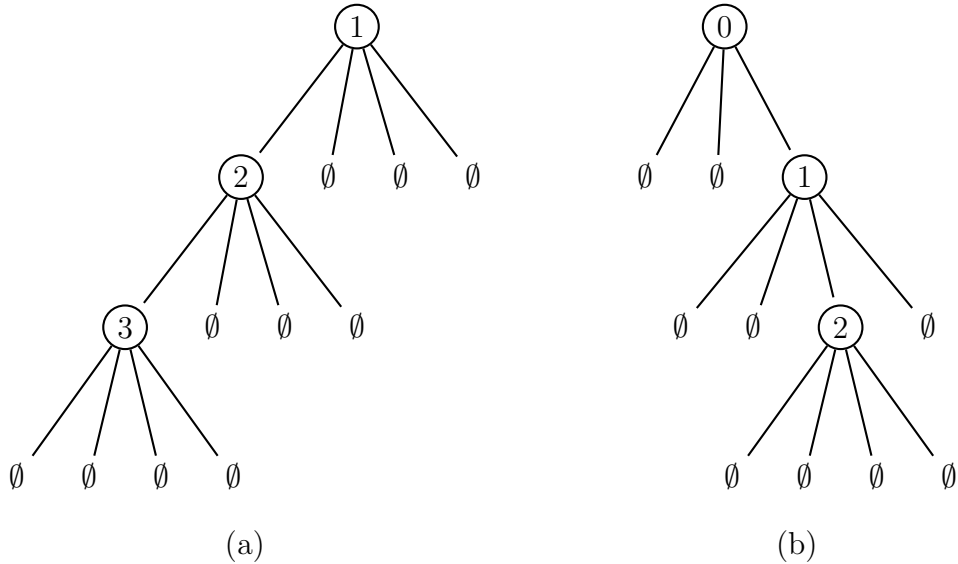

\centering
\begin{tabular}{cc}
  \pstree[nodesepB=3pt]{\Tcircle{1}}{%
    \pstree[nodesepB=3pt]{\Tcircle{2}}{%
      \pstree[nodesepB=3pt]{\Tcircle{3}}{%
         \TR{$\emptyset$}
         \TR{$\emptyset$}
         \TR{$\emptyset$}
         \TR{$\emptyset$}
      }
      \TR{$\emptyset$}
      \TR{$\emptyset$}
      \TR{$\emptyset$}
    }
    \TR{$\emptyset$}
    \TR{$\emptyset$}
    \TR{$\emptyset$}
  }
& \qquad \qquad
  \pstree[nodesepB=3pt]{\Tcircle{0}}{%
    \TR{$\emptyset$}
    \TR{$\emptyset$}
    \pstree[nodesepB=3pt]{\Tcircle{1}}{%
      \TR{$\emptyset$}
      \TR{$\emptyset$}
      \pstree[nodesepB=3pt]{\Tcircle{2}}{%
         \TR{$\emptyset$}
         \TR{$\emptyset$}
         \TR{$\emptyset$}
         \TR{$\emptyset$}
      }
      \TR{$\emptyset$}
    }
  }
\\
\mbox{} & \\
(a) &\qquad\qquad (b) \\
\end{tabular}
\caption{
   (a) The $(3,[3])$--Stirling permutation $333222111$ and its
   corresponding $4$--ary increasing $[3]$--tree.
   (b) The $(3,2,[2])$--Stirling permutation $0\underline{0}1\underline{1}2221$
   and its corresponding $(4,3)$--ary increasing $[2]$--tree. This
   permutation has two ascents at indices 2 and
   4. These ascents are underlined in the permutation for clarity.
} \label{Fig1}
\end{figure}
%%%%%%%%%%%%%%%%%%%%%%%%%%%%%%%%%%%%%%%%%%%%%%%%%%%%%%%%%%%%%%%%%%%%%%%%%%%%

\begin{definition} \label{def increasing s_forest}
Let $n\geq0$,  $s\geq 1$ be integers. An $(s,[n])$--forest is an ordered 
forest $\bm{F}=(T_1,\ldots,T_s)$ composed by $s$
labelled generalized increasing $X_i$--trees $T_i$, for some generalized ordered
partition $(X_1,\ldots, X_n)$ of $[n]$ (where we allow $X_i=\varnothing$).
Given $\mathbf{u}=(u_1,\ldots,u_s)$ with $u_i\geq 1$ integers, a
$(d,\mathbf{u})$--ary increasing
$(s,[n])$--forest $\bm{F}=(T_1,\ldots,T_s)$ is an $(s,[n])$--forest such each $T_i$ is
a $(d,u_i)$-ary increasing $X_i$-tree, for some
generalized partition $(X_1,\ldots,X_s)$ of $[n]$.
\end{definition}

\begin{theorem} \label{theo.ary increasing_forest}
Let $\nu, s\geq 1$,  $t\geq 0$ be integers, $\bm{t}=(t_1,\ldots,t_s)$ a
generalized ordered partition of $t$ ($t_i\geq0$), and
$\mathbf{1}=(1,\ldots,1)$. The family of
$(\nu+1,\bm{t}+\mathbf{1})$--ary increasing $(s,[n])$--forests is
in natural bijection with the class of $(\nu,\bm{t},n)$--Stirling permutations.
\end{theorem}

\begin{proof}
This is a straightforward generalization of
Theorem~\ref{theo.nu+1.t+1_ary increasing}.
See Figure \ref{Fig2} for a concrete example of this class of forests.
\end{proof}

\medskip

\noindent
{\bf Remarks.} 1. It is important to stress that $k$ ascents in a
Stirling permutation correspond to $k$ ``non-leftmost'' internal nodes
in the corresponding tree/forest representation. See Figures~\ref{Fig1}
and~\ref{Fig2}. Hereafter we will say that a tree/forest has an ascent if
the corresponding generalized Stirling permutation has an ascent. 

2. Park \cite{Park_94a} gives two bijections for
the class of $(\nu,[n])$--Stirling permutations: one in terms of $(\nu+1)$-ary
increasing trees, and another one in terms of (ordered) forests of increasing
trees. We have adapted the former for the class of $(\nu,t,[n])$--Stirling
permutations, but we will not use the latter in the present paper.

%%%%%%%%%%%%%%%%%%%%%%%%%%%%%%%%%%%%%%%%%%%%%%%%%%%%%%%%%%%%%%%%%%%%%%%%%%%
%
% FIGURE 2:
%
\begin{figure}[htb]
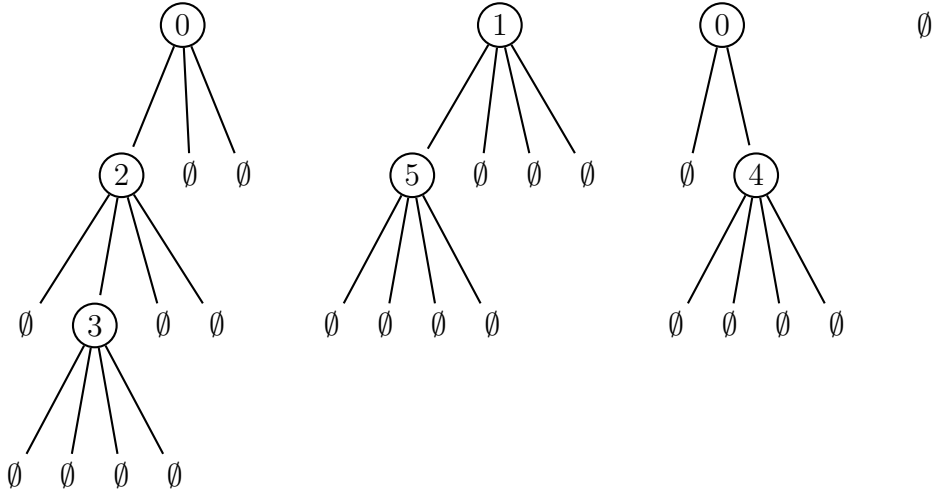

\centering
  \pstree[nodesepB=3pt,treesep=0.5cm]{\Tcircle{0}}{%
    \pstree[nodesepB=3pt,treesep=0.5cm]{\Tcircle{2}}{%
      \TR{$\emptyset$}
      \pstree[nodesepB=3pt,treesep=0.5cm]{\Tcircle{3}}{%
         \TR{$\emptyset$}
         \TR{$\emptyset$}
         \TR{$\emptyset$}
         \TR{$\emptyset$}
      }
      \TR{$\emptyset$}
      \TR{$\emptyset$}
    }
    \TR{$\emptyset$}
    \TR{$\emptyset$}
  }
\qquad
  \pstree[nodesepB=3pt,treesep=0.5cm]{\Tcircle{1}}{%
    \pstree[nodesepB=3pt,treesep=0.5cm]{\Tcircle{5}}{%
      \TR{$\emptyset$}
      \TR{$\emptyset$}
      \TR{$\emptyset$}
      \TR{$\emptyset$}
    }
    \TR{$\emptyset$}
    \TR{$\emptyset$}
    \TR{$\emptyset$}
  }
\qquad
  \pstree[nodesepB=3pt,treesep=0.5cm]{\Tcircle{0}}{%
    \TR{$\emptyset$}
    \pstree[nodesepB=3pt,treesep=0.5cm]{\Tcircle{4}}{%
      \TR{$\emptyset$}
      \TR{$\emptyset$}
      \TR{$\emptyset$}
      \TR{$\emptyset$}
    }
  }
\qquad
\pstree[nodesepB=3pt,treesep=0.5cm]{\Tr{$\emptyset$}}{\Tn}
\vspace*{2mm}
\caption{
The $(3,\bm{t},5)$--Stirling permutation
$(\underline{2}3332200,555111,\underline{0}444,\emptyset)$
corresponding to $\bm{t}=(2,0,1,0)$ and the generalized partition
$(\{2,3\},\{1,5\},\{4\},\emptyset)$ of $[5]$. We show the corresponding
$(4,\bm{t}+\mathbf{1})$--ary $(4,[5])$--forest $\bm{F}=(T_1,T_2,T_3,T_4)$.
}\label{Fig2}
\end{figure}
%%%%%%%%%%%%%%%%%%%%%%%%%%%%%%%%%%%%%%%%%%%%%%%%%%%%%%%%%%%%%%%%%%%%%%%%%%%%
 
%%%%%%%%%%%%%%%%%%%%%%%%%%%%%%%%%%%%%%%%%%%%%%%%%%%%%%%%%%%%%%%%%%
%
% nu-th ORDER (s,t)-EULERIAN FAMILY
%
\section{\texorpdfstring{The $\nu$-order $(s,t)$-Eulerian numbers}%
                        {The nu-order (s,t)-Eulerian numbers}}
\label{sec.nu_eulerian}

We study in detail some properties of the $\nu$-order $(s,t)$-Eulerian numbers
introduced in Definition~\ref{def_gen_Eulerian}, and whose combinatorial
interpretations have been discussed in Theorems~\ref{theo.Eulerian_nu_st}
and~\ref{theo.ary increasing_forest} (see Remark~1 after 
Theorem~\ref{theo.ary increasing_forest}). In this section, 
$s,t$ will be considered indeterminate parameters.
These numbers satisfy the recurrence
relation \eqref{def_recurrence_stnuEulerianOK} which is a particular case
of the one analyzed in \cite{BSV}. By using a generating-function approach, 
that yields a first-order linear partial differential equation which is 
solved with the method of characteristics \cite{BSV}, we obtain  
the exponential generating function (EGF)
\begin{equation}
F^{(\nu)}(x,y;s,t)=\sum_{n,k\geq0} \nueulergen{n}{k}{(\nu)}{(s,t)}
   x^k\frac{y^n}{n!}
\end{equation}
for the $\nu$-order $(s,t)$-Eulerian numbers is given by
\cite[Section~A.1.5]{BSV}:
\begin{equation}
F^{(\nu)}(x,y;s,t) \;=\;
          \left( \frac{ T_\nu\left(e^{y(1-x)^\nu} \, T^{-1}_\nu(x) \right)}
                         {x} \right)^s \,
             \left( \frac{1-x}
                    {1 -  T_\nu\left( e^{y(1-x)^\nu} T^{-1}_\nu(x) \right)}
             \right)^{s+t} \,,
\label{def_EGF_stnuEulerian}
\end{equation}
where $T_\nu$ ($\nu\in\N$) is a one-parameter family of functions given by
\begin{equation}
T_\nu^{-1}(z) \;=\; z \, e^{ Q_\nu(z) } \,, \quad \text{where} \quad
Q_\nu(z)      \;=\; \sum\limits_{k=1}^{\nu-1} \binom{\nu-1}{k} \,
                    \frac{(-z)^k}{k} \,.
\label{def_Tnu}
\end{equation}
For $\nu=1$, $T_1=\bone$ is the identity function, and for $\nu=2$, $T_2$ is
the tree function $T_2=T$ \cite{Corless_96,Corless_97}.

The {\em $\nu$-order $(s,t)$-Eulerian polynomials} are defined as:
\begin{equation}
P^{(\nu)}_n(x;s,t) \;=\; \sum\limits_{k=0}^n \nueulergen{n}{k}{(\nu)}{(s,t)}
     \, x^k
\label{def_stnuEulerian_poly}
\end{equation}
and satisfy that $P^{(\nu)}_n(1;s,t)$ is given by the number of 
$(\nu,t,n)$-Stirling permutations stated in \eqref{eq.number_nu_s_t_n}.
They can be computed by using Theorem~4.1 and Eq.~(4.4) of Ref.~\cite{BSV}:
\begin{subequations}
\label{eq_stnuEulerian_poly}
\begin{align}
P^{(\nu)}_n(x;s,t)  &=  \frac{(1-x)^{s+t+\nu n}}{x^s}\,
                   \frac{n!}{2\pi i} \, \int _{C}
                   \frac{z^{s-1}}{(1-z)^{s+t+1-\nu}}
                   \left[ \log\frac{ze^{Q_\nu(z)}}{xe^{Q_\nu(x)}}
                   \right]^{-n-1} \!\!dz
\label{eq_stnuEulerian_poly1}    \\[2mm]
&= \frac{(1-x)^{s+t+\nu n}}{x^s}\, \lim_{z\rightarrow x}
\frac{\partial^n}{\partial z^n}
\left( \frac{z^{s-1}(z-x)^{n+1}}{(1-z)^{s+t+1-\nu}}
\left[\log\frac{ze^{Q_\nu(z)}}{xe^{Q_\nu(x)}}\right]^{-n-1} \right) \,,
\label{eq_stnuEulerian_poly3}
\end{align}
\end{subequations}
where $C$ is a closed simple curve of index $+1$ surrounding only the
singularity at $z = x$ in the complex $z$-plane.

A Rodrigues-like formula for the $\nu$-order $(s,t)$-Eulerian polynomials
can also be obtained from the integral \eqref{eq_stnuEulerian_poly1}
by performing the change of variables $z e^{Q_\nu(z)} = e^u$ and
$xe^{Q_\nu(x)}=e^v$. Therefore, $z=T_\nu(e^u)$ and $x=T_\nu(e^v)$.
We immediately obtain from \eqref{eq_stnuEulerian_poly1}:
\begin{equation}
P^{(\nu)}_n(T_\nu(e^v);s,t) \;=\;
  \frac{(1-T_\nu(e^v))^{s+t+\nu n}}{T_\nu(e^v)^s} \,
  \frac{d^n}{dv^n}\frac{T_\nu(e^v)^s}{(1-T_\nu(e^v))^{s+t}} \,,
\end{equation}
where actual computations are facilitated by the fact that the derivative of
$T_\nu(x)$ is given in closed form by the expression
\begin{equation}
T'_\nu(x) \;=\; \frac{T_\nu(x)}{x\, (1-T_\nu(x))^{\nu-1}} \,.
\label{def_dTnu}
\end{equation}
An equivalent representation of $P^{(\nu)}_n(x;s,t)$
can be obtained directly from the EGF \eqref{def_EGF_stnuEulerian}
after performing the change of variables $y \mapsto u = (1-x)^\nu y$:
\begin{equation}
P^{(\nu)}_n(x;s,t) \;=\; n! \, \frac{(1-x)^{s+t+\nu n} }{x^s} \,
   [u^n] \frac{ \left( T_\nu\left(e^u \, T^{-1}_\nu(x) \right)
                \right)^s }
              {\left( 1- T_\nu\left( e^u \, T^{-1}_\nu(x) \right)
               \right)^{s+t}} \,.
\label{eq_stnuEulerian_poly4}
\end{equation}
We now illustrate the use of the previous results to derive explicit
expressions for the $(s,t)$-Eulerian numbers and the second order
$(s,t)$-Eulerian numbers. In fact, one can use similar techniques to obtain
formulas for higher-order $(s,t)$-Eulerian numbers, although these
computations are more involved.

%%%%%%%%%%%%%%%%%%%%%%%%%%%%%%%%%%%%%%%%%%%%%%%%%%%%%%%%%%%%%%%%%%
%
% (s,t)-EULERIAN FAMILY
%
\section{\texorpdfstring{The $(s,t)$-Eulerian numbers}
                        {The (s,t)-Eulerian numbers}} \label{sec.st_eulerian}

When $\nu=1$, we will employ the traditional notation
$A_n^{(s,t)}(x)= P^{(1)}_n(x;s,t)$. By using \eqref{eq_stnuEulerian_poly4},
we immediately get
\begin{equation}
A_n^{(s,t)}(x) \;=\; (1-x)^{s+t+n}\, n!\, [u^n]
                     \frac{e^{s\, u}}{(1-x \, e^u)^{s+t}}\,.
\label{eq_stEulerian_poly5}
\end{equation}
This formula allows us to obtain the following closed expressions
for $A_n^{(s,t)}$:
\begin{align}
A_n^{(s,t)}(x) &\;=\; (1-x)^{s+t+n} \sum\limits_{j\ge 0}
                   \frac{ (s+t)^{\overline{j}} }{j!} \, (s+j)^n \, x^j
\label{eq_stEulerian_poly7}\\
  &\;=\; \sum\limits_{k\ge 0} x^k \sum\limits_{j=0}^k
                   (-1)^{k-j} \,
                   \frac{(n+s+t)^{\underline{k-j}}}{j!\, (k-j)!} \,
                   (s+t)^{\overline{j}} \, (s+j)^n \,,
\label{eq_stEulerian_poly8}
\end{align}
where $x^{\overline{j}}=x(x+1)\cdots(x+j-1)$ and 
$x^{\underline{j}}=x(x-1)\cdots(x-j+1)$ are the raising and falling factorials,
respectively. From \eqref{eq_stEulerian_poly7}, we easily obtain

\begin{proposition} \label{prop.stEulerian}
The $(s,t)$-Eulerian polynomials $A_n^{(s,t)}$ satisfy the relation
\begin{equation}
\frac{ x \, A_n^{(s,t)}(x) }{ (1-x)^{n+s+t} } \;=\;
 \sum\limits_{k\ge 1} \frac{ (s+t)^{\overline{k-1}} }{(k-1)!} \,
       (k + s -1)^n \, x^k \,,
\end{equation}
for any $n\ge 0$ and arbitrary parameters $s,t$.
\end{proposition}

This proposition generalizes the well-known formulas for the ordinary
Eulerian polynomials $A_n=A^{(1,0)}_n$:
\begin{equation}
\frac{ x \, A_n(x) }{ (1-x)^{n+1} } \;=\; \sum\limits_{k\ge 1} k^n \, x^k \,,
\label{eq.ratio_Eulerian.poly}
\end{equation}
and for the Eulerian polynomials with the traditional indexing $A^{(0,1)}_n$
\cite[Theorem~1.21]{Bona_12}:
\begin{equation}
\frac{A^{(0,1)}_n(x) }{ (1-x)^{n+1} } \;=\; \sum\limits_{k\ge 0} k^n \, x^k \,.
\label{eq.ratio_Eulerian.poly_Bis}
\end{equation}

A closed expression for $\eulergen{n}{k}{(s,t)}$ can be obtained from
\eqref{eq_stEulerian_poly8} to conclude that

\begin{theorem} \label{theo.stEulerian}
The generalized $(s,t)$-Eulerian numbers  are equal to
\begin{equation}
\eulergen{n}{k}{(s,t)} \;=\;\nueulergen{n}{k}{(1)}{(s,t)}
                       \;=\; \frac{1}{k!} \sum\limits_{j=0}^k
  (-1)^{k-j} \, \binom{k}{j} \, (n+s+t)^{\underline{k-j}} \,\,
   (s+t)^{\overline{j}} \, (s+j)^n
\label{eq.stEulerian}
\end{equation}
for $n\ge 0$ and $0\le k\le n$.
\end{theorem}

\medskip

\noindent
{\bf Remarks.} 1. It is obvious in Eq.~\eqref{eq.stEulerian} that
the numbers $\eulergen{n}{k}{(s,t)}$ are polynomials in both parameters
$s,t$.

2. The ordinary Eulerian numbers  with the standard
\cite[Eq.~(6.38)]{Graham_94}  and  the  traditional \cite{Comtet_74}
ordering are respectively given by
\begin{subequations}
\begin{align}
\euler{n}{k} &\;=\; \eulergen{n}{k}{(1,0)} \;=\; \sum\limits_{j=0}^k
  (-1)^j \, \binom{n+1}{j} \,  (k-j+1)^n \,,
\label{eq.Eulerian}
\\
A(n,k) &\;=\; \eulergen{n}{k}{(0,1)} \;=\; \sum\limits_{j=0}^k
  (-1)^j \, \binom{n+1}{j} \,  (k-j)^n  \;=\; \euler{n}{k-1} \,.
\label{eq.Eulerian_Bis}
\end{align}
\end{subequations}

3. The \textit{shifted} $r$-Eulerian numbers corresponding to $(s,t)=(r,0)$ are
a natural generalization of the $r$-Eulerian numbers \cite[p.~215]{Riordan_58},
\cite[Chapter~II, p.~17]{Foata_70}, \cite{Magagnosc_80},
\cite[Problems~17 and~18, p.~38]{Bona_12} that fit in the framework of
the problem discussed in Ref.~\cite{BSV}.

4. Notice that the $(s,-s)$-Eulerian numbers take the simple form
(cf. \eqref{def_EGF_stnuEulerian}):
\begin{equation}
\eulergen{n}{k}{(s,-s)} \;=\; (-1)^k \, \binom{n}{k} \, s^n \,.
\end{equation}

%%%%%%%%%%%%%%%%%%%%%%%%%%%%%%%%%%%%%%%%%%%%%%%%%%%%%%%%%%%%%%%%%%
%
% Second.order (s,t)-EULERIAN FAMILY
%
\section{\texorpdfstring{The second order $(s,t)$-Eulerian numbers}
                        {The second order (s,t)-Eulerian numbers}} 
\label{sec.second_st_eulerian}

When $\nu=2$, it is customary to write $B_n^{(s,t)}(x)=P^{(2)}_n(x;s,t)$.
By using (\ref{eq_stnuEulerian_poly4}) we immediately get
\begin{equation}
B_n^{(s,t)}(x) \;=\;
n!\, \frac{(1-x)^{s+t+2n}}{x^s} \,
        [u^n] \frac{\left(T\left( T^{-1}(x)\, e^u \right)\right)^s}
              {\left( 1 - T\left( T^{-1}(x)\, e^u \right) \right)^{s+t}}
\label{eq_2nd_stEulerian_poly5} \,,
\end{equation}
where $T$ is the tree function
\cite{Corless_96,Corless_97}. For $|z|< e^{-1}$, this function is given by
the power series:
\begin{equation}
T(z) \;=\; \sum\limits_{n=1}^\infty \frac{ n^{n-1}}{n!} \, z^n \,.
\label{def_Tree}
\end{equation}
It satisfies that $T(z)\exp(-T(z))=z$ (or equivalently,
$T^{-1}(z)=z e^{-z}$), and it is closely related to the Lambert $W$ function
\cite{Corless_96,Corless_97}: $T(z) = - W(-z)$.

Using \eqref{eq_2nd_stEulerian_poly5}, it is not very difficult to obtain an
explicit closed form for both the second-order $(s,t)$-Eulerian polynomials
and the second-order $(s,t)$-Eulerian numbers. By expanding
$(1-T(\xi))^{-(s+t)}$ in powers of $T(\xi)$, where
$\xi=T^{-1}(x) \,e^u = x\, e^{-x+u}$, we get:
\begin{equation}
B_n^{(s,t)}(x) \;=\; n! \, \frac{(1-x)^{s+t+2n}}{x^s} \,
  \sum\limits_{j=0}^\infty
  \frac{ (s+t)^{\overline{j}}}{j!} \, [u^n] T(\xi)^{s+j} \,.
\label{eq_2nd_stEulerian_poly7}
\end{equation}

One important property of the tree function \eqref{def_Tree} is that the
Taylor expansion at $z=0$ of its powers can be computed in closed form
\cite[Eq.~(10)]{Corless_97}:
\begin{equation}
T(z)^s \;=\; \sum\limits_{k=0}^\infty  \frac{s\, (k+s)^{k-1}}{k!}\, z^{s+k}\,.
\label{def_power_Tree}
\end{equation}
Using this expression in \eqref{eq_2nd_stEulerian_poly7} we obtain
\begin{multline}
B_n^{(s,t)}(x) \;=\;  (1-x)^{s+t+2n} \,
\sum\limits_{p=0}^\infty \frac{x^p}{p!} e^{-x(p+s)} \\
   \times \sum\limits_{j=0}^p
  \binom{p}{j}\, (s+t)^{\overline{j}}\, (s+j)\, (p+s)^{n+p-j-1}\,.
\label{eq_2nd_stEulerian_poly9}
\end{multline}

From this equation we easily obtain the following
proposition, which resembles Proposition~\ref{prop.stEulerian} for the
$(s,t)$-Eulerian polynomials $A_n^{(s,t)}$:

\begin{proposition} \label{prop.2nd_stEulerian}
The second-order $(s,t)$-Eulerian polynomials $B_n^{(s,t)}$
satisfy for any $n\ge 0$ and arbitrary parameters $s,t$ the relation
\begin{multline}
\frac{ x e^{x(s-1)}\, B_n^{(s,t)}(x) }{ (1-x)^{2n+s+t} }\;=\;  \,
 \sum\limits_{k\ge 1} \frac{\left(x e^{-x}\right)^k}{(k-1)!} \\[2mm]
 \times
 \sum\limits_{j=0}^{k-1} \binom{k-1}{j} \, (s+t)^{\overline{j}}\, (s+j) \,
       (k + s -1)^{n+k-j-2} \,.
\end{multline}
\end{proposition}

When $(s,t)=(1,0)$, we get the following relation for the ordinary
second-order Eulerian polynomials $B_n(x)=B_n^{(1,0)}(x)$:
\begin{equation}
\frac{ x \, B_n(x) }{ (1-x)^{2n+1} } \;=\; \sum\limits_{k\ge1}
\frac{ k^{n+k-1}}{(k-1)!} \, \left(x e^{-x}\right)^k \,,
\end{equation}
that resembles Eq.~\eqref{eq.ratio_Eulerian.poly} for the ordinary
Eulerian polynomials $A_n(x)$. The proof of these results makes use of the
following combinatorial identities:
\begin{equation}
1 \;=\;
\sum\limits_{j=0}^n \binom{n}{j}\, j! \, j \, \frac{1}{n^{j+1}} \;=\;
\sum\limits_{j=0}^n \binom{n}{j}\, (j+1)! \, \frac{1}{(n+1)^{j+1}} \,.
\label{eq.lemma1}
\end{equation}

A closed expression for the second-order generalized $(s,t)$-Eulerian numbers
can be obtained by writing \eqref{eq_2nd_stEulerian_poly9} in the form
\begin{eqnarray}
B_n^{(s,t)}(x) &=&  \sum\limits_{k\ge 0} \frac{x^k}{k!} \,
                    \sum\limits_{r=0}^k \binom{k}{r} \,
                    (s+t+2n)^{\underline{k-r}}
                    \sum\limits_{p=0}^r \binom{r}{p} (-1)^{k-p}
                    \nonumber \\
  & & \qquad \qquad \times
                    \sum\limits_{j=0}^p \binom{p}{j} (s+t)^{\overline{j}}\,
                    (s+j)\, (p+s)^{n+r-j-1} \,,
\end{eqnarray}
to conclude

\begin{theorem} \label{theo.2nd_stEulerian}
The second-order generalized $(s,t)$-Eulerian numbers are equal to
\begin{multline}
\eulersecondgen{n}{k}{(s,t)} \;=\; \nueulergen{n}{k}{(2)}{(s,t)}
                             \;=\;  \frac{1}{k!} \,
                    \sum\limits_{r=0}^k \binom{k}{r} \,
                    (s+t+2n)^{\underline{k-r}}
                    \sum\limits_{p=0}^r \binom{r}{p} (-1)^{k-p}
                    \\
  \times     \sum\limits_{j=0}^p \binom{p}{j} (s+t)^{\overline{j}}\,
              (s+j)\, (p+s)^{n+r-j-1}
\label{eq.2nd_stEulerian}
\end{multline}
for $n\ge 0$ and $0\le k\le n$.
\end{theorem}

\medskip

\noindent
{\bf Remarks.} 1. Again, from Eq.~\eqref{eq.2nd_stEulerian} we see that
the numbers $\eulersecondgen{n}{k}{(s,t)}$ are polynomials in both parameters
$s,t$.

2. The  ordinary second-order Eulerian numbers  with the standard
\cite[Eq.~(6.38)]{Graham_94}  and  the  traditional \cite{Comtet_74}
ordering are respectively given by
\begin{subequations}
\begin{align}
\eulersecond{n}{k} &\;=\; \eulersecondgen{n}{k}{(1,0)} \;=\;
\sum\limits_{r=0}^k
 (-1)^{k-r} \, \binom{1+2n}{k-r} \, \stirlingsubset{n+r+1}{r+1} \,,
\label{eq.2nd_Eulerian}
\\
B_{n,k} &\;=\; \eulersecondgen{n}{k}{(0,1)} \;=\;
   \sum\limits_{r=0}^k
   (-1)^{k-r} \, \binom{1+2n}{k-r} \, \stirlingsubset{n+r}{r} \;=\;
   \eulersecond{n}{k-1} \,,
\label{eq.2nd_Eulerian_Bis}
\end{align}
\end{subequations}
where the numbers $\stirlingsubset{n}{k}$ are the standard
Stirling subset numbers \cite{Graham_94}. The inverse relation of
Eq.~\eqref{eq.2nd_Eulerian} is given in \cite[Eq.~(6.43)]{Graham_94}.

3. The second-order $(s,-s)$-Eulerian numbers take the form
\begin{equation}
\eulersecondgen{n}{k}{(s,-s)} \;=\; s \sum\limits_{r=0}^k \frac{1}{r!}
\, \binom{2n}{k-r} \sum\limits_{p=0}^r \binom{r}{p} \, (-1)^{k-p} \,
(p+s)^{n+r-1} \,.
\label{eq.2nd_s-s-Eulerian}
\end{equation}

%%%%%%%%%%%%%%%%%%%%%%%%%%%%%%%%%%%%%%%%%%%%%%%%%%%%%%%%%%%%%%%%%%%%%%%
%
% nu-order Ward numbers
%
\section{\texorpdfstring{The $\nu$-order generalized $(s,t)$-Ward numbers}%
                        {The nu-order generalized (s,t)-Ward numbers}} 
\label{sec.nu_ward}

The standard Ward numbers and the second-order Eulerian numbers form an 
inverse Riordan pair \eqref{SecondOrderEuler_vs_Ward}. A natural question is
how to generalize these Ward numbers such that they form a Riordan inverse  
pair with the $\nu$-order $(s,t)$-Eulerian numbers. To achieve this goal, we
start by defining:  

\begin{definition} \label{def_gen_Ward}
Let $\nu,s\geq 1$ and $t\geq 0$ be integers.
The $\nu$-order generalized $(s,t)$-Ward numbers $W^{(\nu)}(n,k;s,t)$
are defined as those satisfying the recurrence
\begin{multline}
W^{(\nu)}(n,k;s,t)  \;=\; (k+s)\, W^{(\nu)}(n-1,k;s,t) \\
+
(\nu n + k + s + t - 1 -\nu) \, W^{(\nu)}(n-1,k-1;s,t)
                           + \delta_{k0}\delta_{n0} \,,
\label{def_recurrence_stnuWardOK}
\end{multline}
with the additional conditions $W^{(\nu)}(n,k;s,t)=0$ if $n<0$ or
$k<0$.
\end{definition}

The family of $\nu$-order generalized $(s,t)$-Ward numbers  is related to the
$\nu$-order $(s,t)$-Eulerian numbers by a \emph{non-trivial} involution
$F\to \widehat{F}$ that can be derived from the following:

\begin{proposition} \label{prop.involution}
Let $F(x,y)=F(x,y;\bm{\mu})$ be the solution of
\begin{equation}
-(\beta + \beta'\, x) \, x  \, \frac{\partial F}{\partial x} +
(1-\alpha \, y - \alpha' \, x \, y  ) \, \frac{\partial F}{\partial y} \\
\;=\; (\alpha + \gamma + (\alpha' + \beta' + \gamma')\, x ) \, F \,,
\label{eq_PDE_final}
\end{equation}
with parameters $\bm{\mu}= (\alpha,\beta,\gamma;\alpha',\beta',\gamma')$,
$\beta\neq 0$, and initial condition $F(x,0)=F(x,0;\bm{\mu})$ $=1$. Then,
\begin{equation}
\widehat{F}(x,y)\;=\; \widehat{F}(x,y; {\bm{\widehat{\mu}}} ) \;=\;
F\left(  \frac{ \beta \, x}{\beta - \beta' \, x} ,
            y\, \frac{\beta - \beta' \, x}{\beta} ;
                        {\bm{\mu}}  \right) \,,
\label{def_fhat}
\end{equation}
is a solution of Eq.~\eqref{eq_PDE_final} with parameters
\begin{equation}
\bm{\widehat{\mu}} \;=\;
\left( \alpha,\beta,\gamma; \alpha'+\beta'-\frac{\alpha\,\beta'}{\beta},
             -\beta',\gamma'+\beta'-\frac{\gamma\,\beta'}{\beta} \right) \,,
\label{def_M2}
\end{equation}
and initial condition
$\widehat{F}(x,0)=\widehat{F}(x,0;{\bm{\widehat{\mu}}} ) =1$.
\end{proposition}

The straightforward proof relies on making the appropriate change
of variables in Eq.~\eqref{eq_PDE_final}, and then regrouping the resulting
terms. Proposition \ref{prop.involution} implies

\begin{corollary} \label{coro.involution}
If $\binomvert{n}{k}$ (resp.\/ $\widehat{\binomvert{n}{k}}$) is the solution
of
\begin{eqnarray}
 \binomvert{n}{k} =
  (\alpha n + \beta k + \gamma)    \, \binomvert{n-1}{k}
  +
  (\alpha' n + \beta' k + \gamma') \, \binomvert{n-1}{k-1} \,+\,
 \delta_{n0}\delta_{k0}
  \label{eq_binomvert}
\end{eqnarray}
 with parameters $\bm{\mu}$ (resp.\/ $\bm{\widehat{\mu}}$), then
\begin{subequations}
\label{def_relations1-2}
\begin{align}
\binomvert{n}{k} &\;=\; \sum\limits_{j=0}^k \widehat{\binomvert{n}{j}} \,
     \binom{n-j}{n-k} \, \left( \frac{\beta'}{\beta} \right)^{k-j} \,,
\label{def_relation1} \\[2mm]
\widehat{\binomvert{n}{k}} &\;=\; \sum\limits_{j=0}^k \binomvert{n}{j} \,
     \binom{n-j}{n-k} \, \left( -\frac{\beta'}{\beta} \right)^{k-j} \,.
\label{def_relation2}
\end{align}
\end{subequations}
\end{corollary}

\medskip

\noindent
{\bf Remark.} Notice that when $\beta\beta'\neq0$ the pair
\begin{subequations}
\begin{align}
a_k &\;=\; \sum\limits_{j=0}^k \widehat{a}_j\,
     \binom{n-j}{n-k} \, \left( \frac{\beta'}{\beta} \right)^{k-j} \,,\\
\widehat{a}_k &\;=\; \sum\limits_{j=0}^k  a_j\,
     \binom{n-j}{n-k} \, \left( -\frac{\beta'}{\beta} \right)^{k-j}\,,
\end{align}
\end{subequations}
is an \textit{inverse pair} in the sense of Riordan \cite{Riordan_68}
(see also \cite{He_07}), and it generates the combinatorial identity
\begin{equation}
 \sum\limits_{i=j}^k (-1)^{i+j}
     \binom{n-i}{n-k} \, \binom{n-j}{n-i} \;=\; \delta_{kj}\,.
\end{equation}

\medskip

According to the results presented in Sections~A.15 and~A.1.6 of
Ref.~\cite{BSV}, the EGF for the $\nu$-order $(s,t)$-Ward numbers
$F_W(x,y;\bm{\widehat{\mu}})$ with
$\bm{\widehat{\mu}}=(0,1,s; \nu, 1, t+s-\nu-1)$ and
the EGF for the $(\nu+1)$-order $(s,t)$-Eulerian numbers
$F_E(x,y;\bm{\mu})$ with $\bm{\mu}=(0,1,s;\nu+1,-1,t-\nu)$ are related by
(cf.~\eqref{def_fhat}):
\begin{subequations}
\label{def_eqs3-4}
\begin{align}
F_W(x,y; \bm{\widehat{\mu}}) &\;=\;
      F_E \left( \frac{x}{1+x}, y\, (1+x); \bm{\mu} \right) \,,
\label{def_eq4} \\[2mm]
F_E(x,y;\bm{\mu}) &\;=\;
      F_W\left( \frac{x}{1-x}, y\, (1-x); \bm{\widehat{\mu}}\right) \,.
\label{def_eq3}
\end{align}
\end{subequations}
If we use \eqref{def_EGF_stnuEulerian}, we obtain from \eqref{def_eq4} the EGF
for the $\nu$-order $(s,t)$-Ward numbers \cite[Section~A.1.6]{BSV}:
\begin{equation}
F_W(x,y) = \frac{T_{\nu+1}\left( e^{y\, (1+x)^{-\nu}} \,
               T^{-1}_{\nu+1}\left(\frac{x}{1+x}\right)\right)^s}
              {\left[1 - T_{\nu+1}\left(e^{y\, (1+x)^{-\nu}} \,
                         T^{-1}_{\nu+1}\left(\frac{x}{1+x}\right)
                                  \right)\right]^{s+t}}
         \, \frac{1}{x^s \ (1+x)^t}\,.
\label{def_EGF_stnuWard}
\end{equation}
Finally, using \eqref{def_relations1-2} we obtain the following

\begin{corollary} \label{coro.ward.euler}
The numbers $\nueulergen{n}{k}{(\nu)}{(s,t)}$ and
$W^{(\nu)}(n,k;s,t)$ are related by the equations
\begin{subequations}
\label{def_relations3-4}
\begin{align}
W^{(\nu)}(n,k;s,t) &\;=\; \sum\limits_{j=0}^k
       \nueulergen{n}{j}{(\nu+1)}{(s,t)} \,\binom{n-j}{n-k} \,,
\label{def_relation3} \\[2mm]
\nueulergen{n}{k}{(\nu+1)}{(s,t)} &\;=\; \sum\limits_{j=0}^k
 (-1)^{k-j} \, W^{(\nu)}(n,j;s,t) \, \binom{n-j}{n-k} \,.
\label{def_relation4}
\end{align}
\end{subequations}
\end{corollary}

Notice that when $(\nu,s,t)=(1,0,1)$ we recover the Ward numbers 
\cite[entry \seqnum{A134991}]{Sloane}
$W^{(1)}(n,k;0,1)=W(n,k)=\associatedstirlingsubset{n+k}{k}$, corresponding
to $\bm{\mu}=(0,1,0;1,1,-1)$. The numbers
$\associatedstirlingsubset{n}{k}$ are the associated Stirling subset numbers
\cite{Fekete_94}, \cite[entry \seqnum{A008299}]{Sloane}. 
Eq.~\eqref{def_relations3-4}
relates these numbers with the second-order $(0,1)$-Eulerian numbers
(i.e., the second-order Eulerian numbers with the traditional indexing
$\eulersecondgen{n}{k}{(0,1)}=B_{n,k}$) in the form mentioned 
in Eq.~\eqref{Ward_SecondOrderEuler}:
\begin{subequations}
\label{def_relations5-6}
\begin{align}
\associatedstirlingsubset{n+k}{k} &\;=\;
           \sum\limits_{j=0}^k B_{n,j} \, \binom{n-j}{n-k} \,,
\label{def_relation5} \\[2mm]
B_{n,k} &\;=\; \sum\limits_{j=0}^k (-1)^{k-j} \,
            \associatedstirlingsubset{n+j}{j} \, \binom{n-j}{k-j} \,.
\label{def_relation6}
\end{align}
\end{subequations}

As $B_{n,k}= \eulersecond{n}{k-1}$ for $n\ge 1$ and $1\le k \le n$, we
can substitute this expression into \eqref{def_relations5-6} and, after some
algebraic manipulations, we arrive at the formulas
\cite[Corollaries~5 and~4]{Smiley_00}:
\begin{subequations}
\label{def_coros_smiley}
\begin{align}
\eulersecond{n}{k} &\;=\; \sum\limits_{j=0}^k (-1)^{k-j} \,
\associatedstirlingsubset{n+j+1}{j+1} \, \binom{n-j-1}{k-j} \,,
\label{def_coro5_smiley} \\
\associatedstirlingsubset{n+k}{k} &\;=\;  \sum\limits_{j=0}^k
\eulersecond{n}{j} \, \binom{n-j-1}{k-j-1} \,.
\label{def_coro6_smiley}
\end{align}
\end{subequations}

%%%%%%%%%%%%%%%%%%%%%%%%%%%%%%%%%%%%%%%%%%%%%%%%%%%%%%%%%%%%%%%%%%%%%%%
%
% Combinatorial interpretation of the nu-order Ward numbers
%
\section{\texorpdfstring{Combinatorial interpretation of the generalized %
                         \\ Ward numbers}
                        {Combinatorial interpretation of the generalized %
                         Ward numbers}} 
\label{sec.nu_ward_combinatorics}

In this section, we will give a combinatorial interpretation of the
$\nu$-order generalized $(s,t)$--Ward numbers
(cf.~\eqref{def_recurrence_stnuWardOK}) based on the identity
\eqref{def_relation3}.

For fixed values of $n$, $k$, $s$, $t$, and a given generalized partition 
$\bm{t}$ of $t$ with $s$ parts,
the interpretation relies on the fact that, to obtain $W^{(\nu)}(n,k;s,t)$, 
we sum over
the number of $(\nu+2,\bm{t}+\mathbf{1})$--ary increasing
$(s,[n])$--forests $\bm{F}$ with $j$ ascents times
$\binom{n-j}{n-k}$. (Recall Remark~1 after 
Theorem~\ref{theo.ary increasing_forest}.) 
In this latter factor, $n-j$
admits a simple interpretation in terms of the set $\mathcal{E}(\bm{F})$ of
internal nodes of $\bm{F}$
that are the first (leftmost) children of their respective parents.
For a tree $T$ with $n$ internal nodes,
the cardinality of this set is denoted by $D_{n,1}=|\mathcal{E}(T)|$
by Janson {\em et al.}\/ \cite{Janson_11}.
We will see that $\binom{n-j}{n-k}$ is closely related to the number of ways
of marking $n-k$ nodes of the set $\mathcal{E}(\bm{F})$.

%%%%%%%%%%%%%%%%%%%%%%%%%%%%%%%%%%%%%%%%%%%%%%%%%%%%%%%%%%%%%%%%%%%%%%%%%%%
%
% FIGURE 3:
%
\begin{figure}[htb]
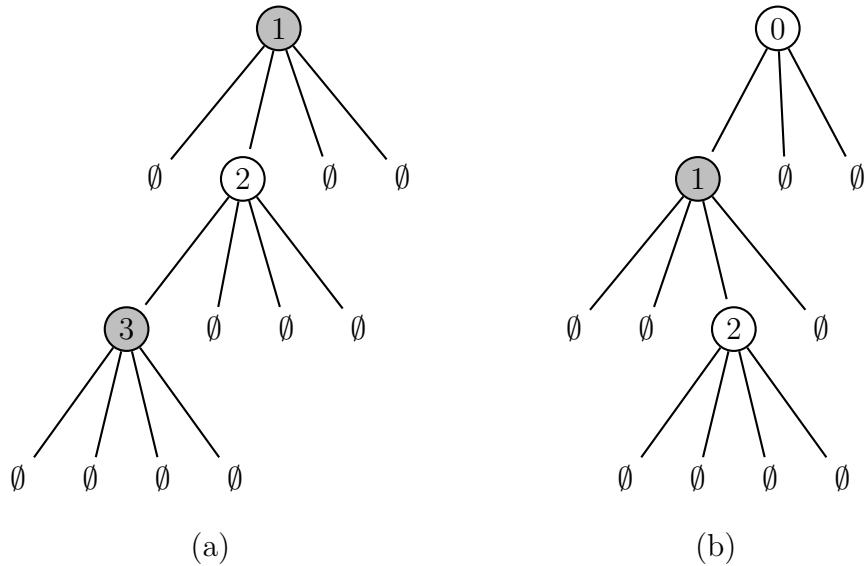

\centering
\begin{tabular}{cc}
\pstree[nodesepB=3pt]{\Tcircle[fillstyle=solid,fillcolor=lightgray]{1}}{%
  \TR{$\emptyset$}
  \pstree[nodesepB=3pt]{\Tcircle{2}}{%
    \pstree[nodesepB=3pt]{\Tcircle[fillstyle=solid,fillcolor=lightgray]{3}}{%
       \TR{$\emptyset$}
       \TR{$\emptyset$}
       \TR{$\emptyset$}
       \TR{$\emptyset$}
    }
    \TR{$\emptyset$}
    \TR{$\emptyset$}
    \TR{$\emptyset$}
  }
  \TR{$\emptyset$}
  \TR{$\emptyset$}
}
& \qquad\qquad
\pstree[nodesepB=3pt]{\Tcircle{0}}{%
  \pstree[nodesepB=3pt]{\Tcircle[fillstyle=solid,fillcolor=lightgray]{1}}{%
    \TR{$\emptyset$}
    \TR{$\emptyset$}
    \pstree[nodesepB=3pt]{\Tcircle{2}}{%
       \TR{$\emptyset$}
       \TR{$\emptyset$}
       \TR{$\emptyset$}
       \TR{$\emptyset$}
    }
    \TR{$\emptyset$}
  }
  \TR{$\emptyset$}
  \TR{$\emptyset$}
}
\\
\mbox{} & \\
(a) &\qquad\qquad (b) \\
\end{tabular}
\caption{(a) A 4--ary increasing $[3]$--tree $T_{\textrm{a}}$, 
  which is equivalent to the
  $(3,[3])$-Stirling permutation $\underline{\bm{1}}\bm{333}222\bm{11}$ with
  one ascent at index 1 (which is underlined) and two distinguished nodes
  (in boldface): the root and the one labelled 3. These nodes are depicted in
  gray. Note that $\mathcal{D}_0(T_{\textrm{a}})=\{$\ding{172},\ding{174}$\}$
  and $|\mathcal{D}_0(T_{\textrm{a}})|=D_{3,1}+1=2$ for this tree.
  (b) A $(4,3)$--ary increasing $[2]$--tree $T_{\textrm{b}}$ equivalent to the
  $(3,2,[2])$--Stirling permutation $\bm{1}\underline{\bm{1}}222\bm{1}00$
  with one ascent at index 2, and one distinguished node (labelled 1).
  In this case,   $\mathcal{D}_1(T_{\textrm{b}})=\{$\ding{172}$\}$ and
  $|\mathcal{D}_1(T_{\textrm{b}})| = D_{2,1}=1$.
  In both examples, all possible distinguishable nodes are actually chosen.
}\label{Fig3}
\end{figure}
%%%%%%%%%%%%%%%%%%%%%%%%%%%%%%%%%%%%%%%%%%%%%%%%%%%%%%%%%%%%%%%%%%%%%%%%%%%

Let us start with the simplest case $s=1$
by considering the class of  $(\nu+2,t+1)$--ary increasing
$[n]$--trees. Then, for any tree $T$ of this class with $j$ ascents,
it is easy to prove that (see \cite[Theorem~2]{Janson_11}):
\begin{equation}
n - j \;=\; |\mathcal{E}(T)| + \delta_{t,0} \,.
\label{def_dn1}
\end{equation}
When $t>0$, we can choose the $n-k$ distinguished nodes from the set
$\mathcal{D}_t(T)=\mathcal{E}(T)$;
when $t=0$, we make the choice from the set $\mathcal{D}_0(T)$ which is now
the union of the root node and the set $\mathcal{E}(T)$.
(Notice that our definition
of ascent slightly differs from that of Ref.~\cite{Janson_11}.) See
Figure~\ref{Fig3} for two examples with $t=0$ (a) and $t>0$ (b). In this
figure, distinguished nodes are depicted in gray.
Putting all together, we can conclude that:

\begin{theorem} \label{theo.ward1}
Let us fix integers $n,t\ge 0$, $\nu\ge 1$, and $0\le k\le n$. Then,
$W^{(\nu)}(n,k;1,t)$ counts the number of $(\nu+2,t+1)$--ary increasing
$[n]$--trees $T$ with at most $k$ ascents and $n-k$ distinguished nodes from
the set $\mathcal{D}_t(T)$.
\end{theorem}

%%%%%%%%%%%%%%%%%%%%%%%%%%%%%%%%%%%%%%%%%%%%%%%%%%%%%%%%%%%%%%%%%%%%%%%%%%%
%
% FIGURE 4:
%
\begin{figure}[htb]
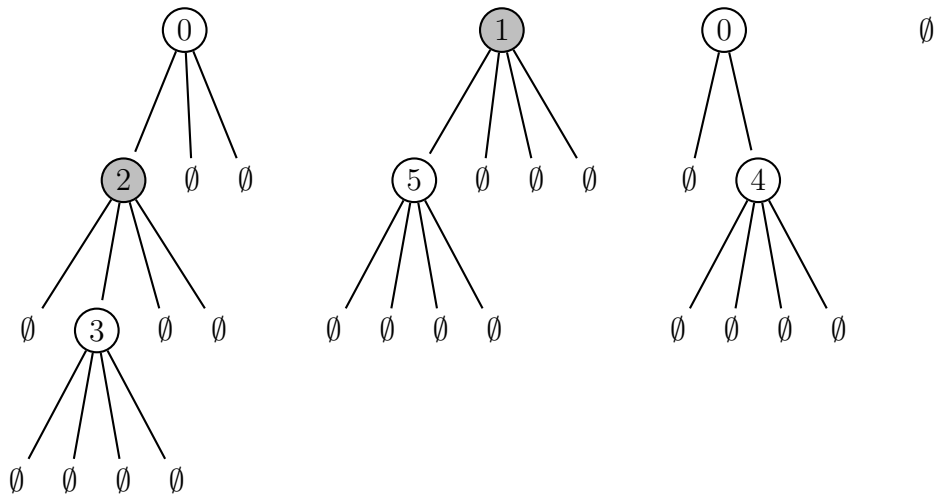

\centering
  \pstree[nodesepB=3pt,treesep=0.5cm]{\Tcircle{0}}{%
    \pstree[nodesepB=3pt,treesep=0.5cm]{%
           \Tcircle[fillstyle=solid,fillcolor=lightgray]{2}}{%
      \TR{$\emptyset$}
      \pstree[nodesepB=3pt,treesep=0.5cm]{\Tcircle{3}}{%
         \TR{$\emptyset$}
         \TR{$\emptyset$}
         \TR{$\emptyset$}
         \TR{$\emptyset$}
      }
      \TR{$\emptyset$}
      \TR{$\emptyset$}
    }
    \TR{$\emptyset$}
    \TR{$\emptyset$}
  }
\qquad
  \pstree[nodesepB=3pt,treesep=0.5cm]{%
         \Tcircle[fillstyle=solid,fillcolor=lightgray]{1}}{%
    \pstree[nodesepB=3pt,treesep=0.5cm]{\Tcircle{5}}{%
      \TR{$\emptyset$}
      \TR{$\emptyset$}
      \TR{$\emptyset$}
      \TR{$\emptyset$}
    }
    \TR{$\emptyset$}
    \TR{$\emptyset$}
    \TR{$\emptyset$}
  }
\qquad
  \pstree[nodesepB=3pt,treesep=0.5cm]{\Tcircle{0}}{%
    \TR{$\emptyset$}
    \pstree[nodesepB=3pt,treesep=0.5cm]{\Tcircle{4}}{%
      \TR{$\emptyset$}
      \TR{$\emptyset$}
      \TR{$\emptyset$}
      \TR{$\emptyset$}
    }
  }
\qquad
\pstree[nodesepB=3pt,treesep=0.5cm]{\Tr{$\emptyset$}}{\Tn}
\vspace*{2mm}
\caption{
A $(4,\bm{t}+\mathbf{1})$--ary increasing $(4,[5])$--forest
$\bm{F}=(T_1,T_2,T_3,T_4)$ with $\bm{t}=(2,0,1,0)$
and the generalized partition $(\{2,3\},\{1,5\},\{4\},\emptyset)$ of $[5]$.
It corresponds to the $(3,\bm{t},5)$--Stirling permutation
$(\underline{\bm{2}}333\bm{22}00,
  555\bm{111},\underline{0}444,\emptyset)$ with $2$ ascents
and two distinguished nodes (labelled $1$ and $2$) out of the three possible
ones. From left to right, the first tree $T_1$ has one ascent at index 1 and
one distinguished node out of $|\mathcal{D}_2(T_1)|=1$;
the second tree has no ascents and one distinguished node out of
$|\mathcal{D}_0(T_2)|=2$; the third tree has one ascent at index 1 and no
distinguished nodes ($\mathcal{D}_1(T_3)=\varnothing$); and the last one, 
$T_4$, is the trivial empty tree.
}\label{Fig4}
\end{figure}
%%%%%%%%%%%%%%%%%%%%%%%%%%%%%%%%%%%%%%%%%%%%%%%%%%%%%%%%%%%%%%%%%%%%%%%%%%%

Let us now consider the extension of Theorem \ref{theo.ward1} for $s\ge 2$.
In this case, our basic objects are indeed the
$(\nu+2,\bm{t}+\mathbf{1})$--ary increasing $(s,[n])$-forests $\bm{F}$ with
$j$ ascents.
Each connected component $T_i$ of the forest $\bm{F}=(T_1,\ldots,T_s)$
is a $(\nu+2,t_i+1)$--ary increasing $X_i$--tree,
where  $(X_1,X_2,\ldots, X_s)$ is a
generalized partition of $[n]$. Let the $j_i$ be the number
of ascents of $T_i$, then  $|X_i| - j_i = |\mathcal{D}_{t_i}(T_i)|$.
If we define the set
\begin{equation}
\mathcal{D}_{\bm{t}}(\bm{F}) \;=\; \mathop{\bigcup}\limits_{i=1}^s
\mathcal{D}_{t_i}(T_i)\,,
\label{def_Fts}
\end{equation}
we get that, irrespectively of the partition $(X_1,X_2,\ldots, X_s)$, for
any $(\nu+2,\bm{t}+\mathbf{1})$--ary
increasing $(s,[n])$-forest $\bm{F}$ with $j$ ascents
\begin{equation}
|\mathcal{D}_{\bm{t}}(\bm{F})| \;=\; \sum_{i=1}^s(|X_i|-j_i)=n-j  \,.
\end{equation}
This fact allows us to generalize Theorem \ref{theo.ward1} when $s\geq 2$:

\begin{theorem} \label{theo.ward2}
Let us fix integers $n,t\ge 0$, $\nu,s\ge 1$, and  $0\le k\le n$. Given
any generalized ordered partition $\bm{t}=(t_1,\ldots,t_s)$ of $t$,
$W^{(\nu)}(n,k;s,t)$ counts the number of $(\nu+2,\bm{t}+\mathbf{1})$--ary
increasing $(s,[n])$--forests $\bm{F}$ with at most $k$ ascents and
$n-k$ distinguished nodes
from the set $\mathcal{D}_{\bm{t}}(\bm{F})$ defined in \eqref{def_Fts}.
\end{theorem}

\medskip

\noindent
This theorem completes the combinatorial interpretation of the
$\nu$--order generalized $(s,t)$--Ward numbers for $\nu,s\ge 1$ and
$t\ge 0$. Figure~\ref{Fig4} shows an example of a
$(4,\bm{t}+\mathbf{1})$--ary increasing
$(4,[5])$--forest $\bm{F}=(T_1,T_2,T_3,T_4)$ with $\bm{t}=(2,0,1,0)$, 
and the generalized partition $(\{2,3\},\{1,5\},\{4\},\emptyset)$ of $[5]$.
This forest has two ascents and two distinguished nodes out of the
three possible ones
$\mathcal{D}_{\bm{t}}(\bm{F})=\{$\ding{172},\ding{173},\ding{176}$\}$.

%%%%%%%%%%%%%%%%%%%%%%%%%%%%%%%%%%%%%%%%%%%%%%%%%%%%%%%%%%%%%%%%%%%%%%%
%
% ACKNOWLEDGMENTS
%
\section*{Acknowledgments}

We are indebted to Alan Sokal for his participation in the early stages
of this work, his encouragement, and useful suggestions later on.
We also thank Jesper Jacobsen, Anna de Mier, Neil Sloane, and Mike Spivey
for correspondence, and David Callan for pointing out some interesting
references to us. Last but not least, we thank Bojan Mohar for
valuable criticisms and suggestions.

%%%%%%%%%%%%%%%%%%%%%%%%%%%%%%%%%%%%%%%%%%%%%%%%%%%%%%%%%%%%%%%%%%%%%%%%%%%%%%
%
% BIBLIOGRAPHY
%

\end{document}